\documentclass[11pt]{amsart}
\usepackage{amsmath, amsthm, amssymb, amscd, amsfonts}

\numberwithin{equation}{section}

\newtheorem{theorem}[equation]{Theorem}
\newtheorem{lemma}[equation]{Lemma}

\newtheorem{proposition}[equation]{Proposition}
\newtheorem{claim}[equation]{Claim}

\theoremstyle{definition}

\theoremstyle{remark}
\newtheorem{remark}[equation]{Remark}

\newcommand{\Z}{{\mathbb Z}}
\newcommand{\B}{{\mathcal B}}
\newcommand{\A}{{\mathcal A}}

\newcommand\id{\operatorname{id}}

\newcommand\coad{\operatorname{coad}}
\newcommand\Tr{\operatorname{Tr}}

\newcommand\co{\operatorname{co}}

\newcommand\Aut{\operatorname{Aut}}

\newcommand\res{\operatorname{res}}

\newcommand\Ind{\operatorname{Ind}}
\newcommand\Supp{\operatorname{Supp}}

\newcommand\PSL{\operatorname{PSL}}

\newcommand\Rep{\operatorname{Rep}}

\begin{document}

\title{Semisimple Hopf algebras of dimension 60}

\author{Sonia Natale}
\address{Facultad de Matem\'atica, Astronom\'\i a y F\'\i sica.
Universidad Nacional de C\'ordoba. CIEM -- CONICET. (5000) Ciudad
Universitaria. C\'ordoba, Argentina} \email{natale@mate.uncor.edu
\newline \indent \emph{URL:}\/ http://www.mate.uncor.edu/natale}

\begin{abstract}We determine the isomorphism classes of semisimple
Hopf algebras of dimension $60$ which are simple as Hopf
algebras.\end{abstract}

\dedicatory{Dedicated to Susan Montgomery}

\thanks{This work was partially supported by CONICET, ANPCyT, SeCYT (UNC), FaMAF and Alexander von Humboldt Foundation.}


\date{November 25, 2009.}

\maketitle

\section{Introduction and main results}
We shall work over an algebraically closed field $k$ of
characteristic zero. Let $H$ be a semisimple Hopf algebra over
$k$. A Hopf subalgebra $K$ of $H$ is normal if it is stable under
the left adjoint action of $H$. If $K$ is normal in $H$, then the
quotient $H/HK^+$ is a Hopf algebra and there is an exact sequence
$k \to K \to H \to \overline H \to k$. In this case, $H$ is
isomorphic to a bicrossed product $K {}^{\tau}\#_{\sigma}\overline
H$ with respect to appropriate compatible data.

The Hopf algebra $H$ is called simple if it contains no proper
normal Hopf subalgebra. The notion of simplicity is self-dual,
that is, $H$ is simple if and only if $H^*$ is simple.

For instance, if $G$ is a finite simple group, then the group
algebra $kG$ and its dual $k^G$ are simple Hopf algebras.
Furthermore, in this case, any twisting deformation of $kG$ is
simple \cite{nik}. However, there are examples of solvable groups
that admit simple twisting deformations \cite{gn}.

It was shown in \cite{ssld} that, up to twisting deformations,
there is no semisimple Hopf algebra of dimension $< 60$ which is
simple as a Hopf algebra. The only simple example in dimension $<
60$ is a twisting of the group $D_3\times D_3$ and has dimension
$36$.

\medbreak In dimension $60$ three examples are known of nontrivial
semisimple Hopf algebras which are simple as Hopf algebras. The
first two are the Hopf algebras $\A_0$ and $\A_1 \simeq \A_0^*$
constructed by Nikshych \cite{nik}. We have $\A_0 = (k\mathbb
A_5)^J$, where $J \in k\mathbb A_5 \otimes k\mathbb A_5$ is an
invertible twist lifted from a nondegenerate $2$-cocycle in a
subgroup of $\mathbb A_5$ isomorphic to $\Z_2 \times \Z_2$.

The third example is the self-dual Hopf algebra $\B$ constructed
in \cite{gn}. In this case $\B = (kD_3\otimes kD_5)^J$, where $J$
is an invertible twist also lifted from a nondegenerate
$2$-cocycle in a subgroup of $D_3\times D_5$ isomorphic to $\Z_2
\times \Z_2$.

As coalgebras, these examples are isomorphic to direct sums of
full matric coalgebras, as follows: \begin{align}\label{coalg-a1}
\A_1 & \simeq k \oplus M_3(k)^{(2)}\oplus M_4(k)\oplus M_5(k),
\\ \label{coalg-a0} \A_0 & \simeq k^{(12)}\oplus M_4(k)^{(3)}, \\ \label{coalg-b}  \B & \simeq
k^{(4)}\oplus M_2(k)^{(6)}\oplus M_4(k)^{(2)}. \end{align} As for
the group-like elements, we have $G(\A_0) \simeq \mathbb A_4$ and
$G(\B) \simeq \Z_2 \times \Z_2$.

\medbreak At the level of fusion categories, it was shown in
\cite[Theorem 9.12]{ENO2} that $\Rep \A_0 \simeq \Rep \mathbb A_5$
is the only simple fusion category of dimension $60$. In this
context, according to  \cite[Definition 9.10]{ENO2}, a fusion
category is called simple if it contains no proper fusion
subcategories. In particular, there are simple Hopf algebras whose
fusion category is not simple in the sense of \cite{ENO2}.

\medbreak As a consequence of  \cite[Corollary 9.14]{ENO2}, it was
shown in \cite[Proposition 6.10]{BN} that if $G(H) = 1$, then $H$
is of type \eqref{coalg-a1} as a coalgebra, furthermore, $H$ is
simple and $H$ is isomorphic to $k^{\mathbb A_5}$ or to $\A_1$.
Also, by \cite[Corollary 6.12]{BN}, if $H$ is simple and of
coalgebra type \eqref{coalg-b}, then $H$ is isomorphic to the
self-dual Hopf algebra $\B$. We shall show in Proposition \ref{a0}
that if $H$ is simple and has coalgebra type \eqref{coalg-a0},
then $H$ is isomorphic to $\A_0 \simeq \A_1^*$.

\medbreak The main result of this paper is the following theorem,
that says that $\A_0$, $\A_1$ and $\B$ are actually the only
nontrivial simple examples in dimension $60$.

\begin{theorem}\label{main} Let $H$ be a nontrivial semisimple Hopf algebra of dimension
$60$. Suppose $H$ is simple. Then $H$ is isomorphic to $\A_0$ or
to $\A_1$ or to $\B$. \end{theorem}

Theorem \ref{main} will be proved in Sections \ref{coal-types} and
\ref{pf}. The proof relies on the main results of the paper
\cite{BN}. In particular, we use the refinement, contained in
\cite[Theorem 1.1]{BN}, of the result  \cite[Theorem 11]{NR} of
Nichols and Richmond on semisimple Hopf algebras with an
irreducible comodule of dimension $2$.

We also make strong use of several tools developed in \cite{ssld}
regarding, for instance, the structure of quotient coalgebras and
relations among the fusion rules for irreducible characters and
Hopf subalgebras. We prove some facts on braided Hopf algebras
over the dihedral groups $D_n$ and the alternating group $\mathbb
A_4$, that we apply to some cases where a reduction to the
Radford-Majid biproduct situation is possible. See Propositions
\ref{braided-dih} and \ref{braided-a4}.

\medbreak The contents of the paper are the following: basic facts
and terminology on semisimple Hopf algebras and their irreducible
characters, as well as some results from \cite{ssld} and
\cite{BN}, are recalled in Section \ref{uno}. We discuss there
some properties about the structure of the left coideal subalgebra
of coinvariants under a Hopf algebra map, that prove to be useful
when considering the different possibilities in dimension $60$.
Section \ref{a4-dn} concerns biproducts, and we consider here some
special cases of braided Hopf algebras over $D_n$ and $\mathbb
A_4$.

Finally, the next two sections  are devoted to the proof of our
main result. In Section \ref{coal-types} we determine the
different possible coalgebra types arising in dimension $60$. See
Proposition \ref{conteo}. These possibilities are then studied
separately in each of the subsections of Section \ref{pf}.

\subsection*{Acknowledgement} It is the author's pleasure to express her recognition
for the influence of Susan Montgomery in her research on
semisimple Hopf algebras, through her own contributions,
interesting discussions and references.

This work was started during a research stay of the author at the
Mathematisches Institut der Universit\" at M\" unchen, as an
Alexander von Humboldt Fellow. She thanks Prof. Hans-J\" urgen
Schneider for the kind hospitality.

\section{Semisimple Hopf algebras}\label{uno}
Let $H$ be a semisimple Hopf algebra over $k$. We next recall some
of the terminology and conventions from \cite{ssld} that will be
used throughout this paper.

\medbreak As a coalgebra, $H$ is isomorphic to a direct sum of
full matrix coalgebras
\begin{equation}\label{estructura} H \simeq k^{(n)} \oplus \oplus_{d_i > 1}
M_{d_i}(k)^{(n_i)},\end{equation} where $n = |G(H)|$. The
Nichols-Zoeller theorem \cite{NZ} implies that $n$ divides both
$\dim H$ and $n_i d_i^2$, for all $i$.

\medbreak If we have an isomorphism as in \eqref{estructura},  we
shall say that $H$ is \emph{of type} $(1, n; d_1, n_1; \dots; d_r,
n_r)$ \emph{as a coalgebra}. If $H^*$ is of type $(1, n; d_1, n_1;
\dots )$ as a coalgebra, we shall say that $H$ is  \emph{of type}
$(1, n; d_1, n_1; \dots )$ \emph{as an algebra}.

So that $H$ is of type $(1, n; d_1, n_1; \dots; d_r, n_r)$ as a
(co-)algebra if and only if $H$ has $n$ non-isomorphic
one-dimensional (co-)representations,  $n_1$ non-isomorphic
irreducible (co-)representations of degree $d_1$, etc.

\medbreak Let $V$ be an $H$-comodule. The \emph{character} of $V$
is the element $\chi = \chi_V \in H$ defined by $\langle f, \chi
\rangle = \Tr_V(f)$, for all $f \in H^*$. For a character  $\chi$,
its degree is the integer $\deg \chi = \epsilon(\chi) = \dim V$.

If $\chi \in H$ is a character, then $\chi$ decomposes as $\chi =
\sum_{\mu} m(\mu, \chi) \mu$, where $\mu$ runs over the set of
irreducible characters of $H$ and $m(\mu, \chi)$ are nonnegative
integers. For all characters $\chi, \psi, \lambda \in H$, we have
\cite{NR}:
\begin{equation}\label{adj}m(\chi, \psi \lambda) = m(\psi^*, \lambda \chi^*) = m(\psi, \chi \lambda^*). \end{equation}

Let $\chi$  be an irreducible character of $H$. The stabilizer of
$\chi$ under left multiplication by elements in $G(H)$ will be
denoted by $G[\chi]$. So that a group-like element $g$ belongs to
$G[\chi]$ if and only if $g \chi = \chi$. By the Nichols-Zoeller
theorem \cite{NZ}, we have that $|G[\chi]|$ divides $(\deg
\chi)^2$.

In view of \cite[Theorem 10]{NR}, $G[\chi]= \{ g \in G(H):\, m(g,
\chi \chi^*) > 0 \} =  \{ g \in G(H):\, m(g, \chi \chi^*) = 1 \}$.
In particular,
\begin{equation*}\chi  \chi^* = \sum_{g \in G[\chi]} g +
\sum_{\deg \mu > 1} m(\mu, \chi \chi^*) \mu.\end{equation*}

The irreducible characters in $H$ span a subalgebra  of $H$, that
coincides with the character algebra $R(H^*)$ of $H^*$. There is a
bijective correspondence between Hopf subalgebras of $H$ and
standard subalgebras of $R(H^*)$, that is, subalgebras spanned by
irreducible characters of $H$. This correspondence assigns to the
Hopf subalgebra $K \subseteq H$ its character algebra $R(K^*)
\subseteq R(H^*)$. See \cite{NR}.

\begin{lemma}\label{psiinB} Suppose $B \subseteq H$ is a Hopf
subalgebra. Let $\chi, \lambda \in B$, $\psi \in H$, be
irreducible characters. If $m(\chi, \psi\lambda) > 0$, then $\psi
\in B$. \end{lemma}

\begin{proof} By \eqref{adj} we have $m(\psi, \chi\lambda^*) = m(\chi, \psi\lambda) >
0$. Since $B$ is a Hopf subalgebra, all irreducible summands of
$\chi\lambda^*$ belong to $B$. Hence $\psi \in B$, as claimed.
\end{proof}

\subsection{Subalgebras of coinvariants} Let $\pi: H \to \overline H$ be a surjective Hopf
algebra map. Consider the subalgebra $H^{\co \pi} \subseteq H$ of
right coinvariants of $\pi$, defined as
$$H^{\co \pi} = \{ h \in H:\, h_{(1)} \otimes \pi(h_{(2)}) = h \otimes 1 \}.$$
We shall also use the notation $H^{\co \overline H} : = H^{\co
\pi}$. This is a left coideal subalgebra of $H$ stable under the
left adjoint action. The map $\pi$ is normal if $H^{\co \pi}$ is a
subcoalgebra, hence a Hopf subalgebra, of $H$.

\begin{remark}\label{cleftness} Let $\pi: H \to \overline H$ be a surjective Hopf
algebra map. By \cite[Theorem 8.2.4 and Proposition 8.4.4]{Mo},
the extension $H^{\co \overline H} \subseteq H$ is $\overline
H$-Galois. On the other hand, $H$ is free over $H^{\co \overline
H}$, by \cite{skryabin}. Hence, by \cite[3.2 (4)]{schneider},
$H^{\co \overline H} \subseteq H$ is an $\overline H$-cleft
extension. In particular, $H$ is isomorphic, as an $\overline
H$-comodule algebra, to a crossed product $H^{\co \overline H}
\#_{\sigma}\overline H$. \end{remark}

\begin{lemma}\label{zetazeta*} Let  $F$ be a finite group. Let also $H$ be a semisimple Hopf algebra
endowed with a surjective Hopf algebra map $\pi: H \to kF$. Then:

\emph{(i)} For all irreducible characters $\zeta \in H$, $m(1,
\pi(\zeta\zeta^*)) \geq \deg \zeta$. Equality holds if and only if
$\pi(\zeta)$ is multiplicity free.

\emph{(ii)} Let $c$ be the least common multiple among the
dimensions of all simple $H^{\co \pi}$-modules. Then, for every
simple $H$-module $U$, $\dim U$ divides the product $c|F|$.

In particular, if $H^{\co \pi}$ is commutative, then $\dim U$
divides the order of $F$, for all simple $H$-modules $U$.
\end{lemma}

\begin{proof} (i) Let $n = \deg \zeta$. Then we may decompose $\pi(\zeta) = \sum_{i =
1}^nx_i$, where $x_i \in F$. Part (i) follows easily from this.

(ii) By cleftness of $H$ as a $kF$-comodule algebra, $H$ is
isomorphic as an algebra to a crossed product $H \simeq R
\#_{\sigma}kF$, where $R = H^{\co \pi}$. See Remark
\ref{cleftness}. The description of the irreducible
representations of a crossed product given in \cite{MoW} implies
that $\dim U$ divides $c|F|$, for all simple $H$-modules $U$.
\end{proof}

\begin{remark}\label{hcopi} Let $\pi: H \to \overline H$ be a surjective Hopf algebra map.
Identify $B^*$ with a Hopf subalgebra of $H^*$ by means of the
transpose of $\pi$. Then, for each irreducible character $\zeta
\in H$, the multiplicity $m(1, \pi(\zeta))$ is exactly $m(k1,
\res^{H^*}_{B^*}\zeta)$, where $\res$ denotes the restriction map.

By Frobenius reciprocity, we have $m(1, \pi(\zeta)) =
m(\Ind_{B^*}^{H^*}k1, \zeta)$.

Consider the decomposition \begin{equation}H^{\co \pi} = k1 \oplus
V_1 \oplus \dots \oplus V_m, \end{equation} into irreducible left
coideals $V_0 = k1, V_1, \dots, V_m$. Let $\zeta_i \in H$ be the
character of $V_i$, $i = 0, \dots, m$.

By \cite[Lemma 1.7.1]{ssld}, $H^{\co \pi} \simeq
\Ind_{B^*}^{H^*}k1$ as left $H$-comodules. Therefore the
multiplicity $m(1, \pi(\zeta_i))$ coincides with the multiplicity
of $V_i$ as a direct summand of $H^{\co \pi}$. \end{remark}

\subsection{The Hopf subalgebras $B[\chi]$}\label{bdechi}
Let $C$ be a simple subcoalgebra of $H$, and let $\chi \in C$ be
the irreducible character of $C$. We shall denote by $B[\chi] : =
k[C\mathcal S(C)]$ the Hopf subalgebra generated by $C\mathcal
S(C)$ as an algebra. Note that $G[\chi] \subseteq B[\chi]$.

The Hopf subalgebra $B[\chi]$ is contained in the adjoint Hopf
subalgebra $H_{\coad}$ of $H$, which is generated by the
irreducible components of $X\mathcal S(X)$, where $X$ runs over
all simple subcoalgebras of $H$. Recall that there is a universal
cocentral exact sequence
\begin{equation}\label{univcocentral}k \to H_{\coad} \to H \to kU\to k
\end{equation}
where  $U$ is the universal grading  group of the category
$H$-comod of finite dimensional $H$-comodules. See
\cite[8.5.]{ENO}, \cite[Theorem 3.8]{gelaki-nik}.

\medbreak The following is one of the main results of \cite{BN}.

\begin{theorem}\label{bn-mres}\cite[Theorem 1.1]{BN}. Suppose $\deg \chi = 2$.
Then $B[\chi]$ is a commutative Hopf subalgebra of $H$ isomorphic
to $k^{\Gamma}$, where $\Gamma$ is a non cyclic finite subgroup of
$\PSL_2(k)$ of even order.

Let $G[\chi] \subseteq G(H)$ be the stabilizer of $\chi$ with
respect to left multiplication. Then we have
\begin{enumerate}\item[(i)] If $|G[\chi]| = 4$, then $B[\chi] \simeq k^{\Z_2\times \Z_2}$.
\item[(ii)] If $|G[\chi]| = 2$, then $B[\chi] \simeq k^{D_n}$, where $n \geq 3$.
\item[(iii)] If $|G[\chi]| = 1$, then $B[\chi] \simeq k^{\mathbb A_4}$,
$k^{\mathbb S_4}$, or $k^{\mathbb A_5}$. \end{enumerate}
\end{theorem}

Theorem \ref{bn-mres} turns out to be useful when discussing low
dimensional semisimple Hopf algebras since, in that case, most
examples would admit irreducible comodules of dimension $2$.

\section{Braided Hopf algebras over $D_n$ and $\mathbb A_4$}\label{a4-dn}

We recall for future use some facts on the Radford-Majid biproduct
construction \cite{majid, radford}.

Let $A$ be a semisimple Hopf algebra and let ${}^A_A\mathcal{YD}$
denote the braided category of Yetter-Drinfeld modules over $A$.
Let $R$ be a semisimple braided Hopf algebra in
${}^A_A\mathcal{YD}$.

Then $R$ is both an algebra and a coalgebra in
${}^A_A\mathcal{YD}$, such that the comultiplication is an algebra
map  in ${}^A_A\mathcal{YD}$ and  the identity map $\id_R$ has a
convolution inverse, called the antipode of $R$. We shall use the
notation $\Delta_R(a) = a^{(1)} \otimes a^{(2)}$ and $\mathcal
S_R$ for the comultiplication and the antipode of $R$,
respectively.

The compatibility between the multiplication and comultiplication
in $R$ is the following:
\begin{equation}\label{c} \Delta_R(ab) = a^{(1)} ((a^{(2)})_{-1}.b^{(1)})
\otimes (a^{(2)})_0b^{(2)}, \end{equation} for all $a, b \in R$.

Let  $H = R \# A$ be the corresponding biproduct; so that $H$ is a
semisimple Hopf algebra with multiplication, comultiplication and
antipode given by
\begin{equation}\label{mca}
(a \# g) (b \# h)  = a (g_{(1)} . b) \# g_{(2)}h,  \quad \Delta(a
\# g) = a^{(1)} \# (a^{(2)})_{-1} g_{(1)} \otimes (a^{(2)})_0 \#
g_{(2)},
\end{equation}
\begin{equation*} \mathcal S (a \# g)  = (1 \# \mathcal S(a_{-1}g)) (\mathcal
S_R(a_0) \# 1), \end{equation*} for all $g, h \in A$, $a, b \in
R$.

\begin{remark}\label{coalg-r} Note that $R = H^{\co \pi}$, where
$\pi = \epsilon_R \otimes \id: H \to A$. Hence $R$ is a normal
left coideal subalgebra of $H$. Moreover, the action of $A$ on $R$
coincides with the restriction of the adjoint action of $H$ to
$A$.

On the other hand, the map $\id \otimes \epsilon: H \to R$ induces
a coalgebra isomorphism $H/HA^+ \simeq R$. The relations
\eqref{mca} imply that the coaction of $A$ on $R$ is given by
$\rho = (\epsilon_R \otimes \id \otimes \id \otimes \id)\Delta: R
\to A \otimes R$.
\end{remark}

\medbreak A biproduct $R \# A$ as described above is characterized
by the following property: suppose $H$ is a finite dimensional
Hopf algebra endowed with Hopf algebra maps $i: A \to H$ and $\pi:
H \to A$, such that $\pi i: A \to A$ is an isomorphism. Then the
subalgebra $R : = H^{\co \pi}$ has a natural structure of
Yetter-Drinfeld Hopf algebra over $A$ such that the multiplication
map $R \# A \to H$ induces an isomorphism of Hopf algebras.

The following lemma gives the existence of Hopf subalgebras in a
biproduct, under appropriate assumptions. Let $R$ be a semisimple
braided Hopf algebra over $A$ and let $H = R \# A$ be the
biproduct.

\medbreak Suppose that $A = kG$, where $G$ is a finite group.
There is a $G$-grading on $R$: $R = \bigoplus_{g \in G}R_g$, which
is both an algebra and a coalgebra grading, \textit{i.e.},
$$R_gR_h \subseteq R_{gh}, \quad \Delta_R(R_g) \subseteq \bigoplus_{st = g}R_s \otimes
R_t,$$ for all $g, h \in G$. This grading corresponds to the
coaction $\rho: R \to kG \otimes R$, $\rho(a) = a_{-1} \otimes
a_0$, in such a way that $R_g = \{ r \in R: \rho(a) = g \otimes a
\}$.

The group $G$ acts on $R$ by algebra and coalgebra automorphisms,
and the action satisfies,  for all $g, h \in G$, the
Yetter-Drinfeld condition
\begin{equation}\label{accion-coaccion}h . R_g =
R_{hgh^{-1}}.\end{equation}

The set $\Supp(R)$ of elements $g \in G$ such that $R_g \neq 0$
will be called the \emph{support} of $R$. If $\Gamma$ is a
subgroup of $G$ containing $\Supp (R)$, then $R$ is a braided Hopf
algebra over $\Gamma$ and $R\#k\Gamma$ is a Hopf subalgebra of $R
\# kG$ \cite[Lemma 4.3.1]{ssld}.

\medbreak The next two propositions will be used later on in our
discussion in the dimension $60$ context.

\begin{proposition}\label{braided-dih} Suppose $n \geq 3$ is odd. Let $R$ be a semisimple braided Hopf algebra over $G =
D_n$ with $\dim R = n+1$. Then $\Supp (R) \subseteq T$, where $T
\simeq \Z_n$ is the subgroup of rotations in $D_n$.

Therefore, $R \# kT$ is a Hopf subalgebra of $R\# kD_n$ of index
$2$. In particular, $R \# kT$ is normal in $R\# kD_n$.
\end{proposition}

\begin{proof} Consider the $D_n$-grading $R = \oplus_{x \in
D_n}R_x$. Assume on the contrary that $\Supp (R) \subsetneq T$.
Then there exists a reflection $s \in D_n$ such that $R_s\neq 0$.

Since $n$ is odd, then the reflections in $D_n$ form a conjugacy
class. In view of the compatibility condition
\eqref{accion-coaccion}, we get that $R_x \neq 0$, for all
reflections $x \in D_n$. Since $D_n$ has exactly $n$ reflections,
$s_1, \dots, s_n$, then we see that $\Supp(R) = \{1, s_1, \dots,
s_n \}$.

Since $\dim R = n+1$, then $\dim R_{s_i} = 1$, for all $i = 0,
\dots, n$, where $s_0 = 1 \in D_n$. Let $u_0 = 1, u_1, \dots, u_n$
be a basis of $R$ such that $u_i \in R_{s_i}$.

\medbreak We have $u_iu_j \in R_{s_is_j}$, for all $1\leq i, j
\leq n$. Since the product $s_is_j$ of two reflections is a
rotation, then $R_{s_is_j} = 0$, for all $1\leq i\neq  j \leq n$,
and therefore $u_iu_j = 0$, for all $1\leq i\neq  j \leq n$. In
particular, $R$ is commutative.

On the other hand, for all $i = 1, \dots, n$, we have $u_i^2 \in
R_{s_i^2} = R_{s_0} = k1$. After rescaling the basis $u_i$, we may
assume that $u_i^2 = 0$ or $1$, for all $i = 1, \dots, n$.

\medbreak Suppose that $u_i^2 = 1$ for some $i \geq 1$, and pick
$j \neq i$, $1\leq j \leq n$. We get $$u_j = (u_i^2)u_j =
u_i(u_iu_j) = u_i0 = 0,$$ which is a contradiction.

Then $u_i^2 = 0$, for all $i = 1, \dots, n$. But this is again a
contradiction, since $R$ is semisimple, by assumption. Therefore
such a grading is impossible, and the proposition follows.
\end{proof}

\begin{proposition}\label{braided-a4} Let $R$ be a semisimple braided Hopf algebra over $\mathbb
A_4$ with $\dim R = 5$. Then $\Supp (R) \subseteq \mathbb K$,
where $\mathbb K \simeq \Z_2 \times \Z_2$ is the Klein subgroup.
Therefore $R \#k\mathbb K$ is a normal Hopf subalgebra of $R \#
\mathbb A_4$ of dimension $20$.
\end{proposition}

\begin{proof} As in the proof of the previous proposition, consider the $\mathbb
A_4$-grading $R = \oplus_{x \in \mathbb A_4}R_x$, and assume on
the contrary that $\Supp (R) \subsetneq \mathbb K$. Then there
exists a $3$-cycle $c \in \mathbb A_4$ such that $R_c\neq 0$.

Condition \eqref{accion-coaccion} implies $R_x \neq 0$, for all
$3$-cycles $x \in \mathbb A_4$ conjugated to $c$. Since the
conjugacy class of a $3$-cycle $c$ in $\mathbb A_4$ has exactly
$4$ elements, $c = c_1, \dots, c_4$, then we get $\Supp(R) = \{1,
c_1, \dots, c_4 \}$. Thus $\dim R_{c_i} = 1$, for all $i = 0,
\dots, 4$, where $c_0 = 1$.

Let, as before, $u_0 = 1, u_1, \dots, u_4$ be a basis of $R$ such
that $u_i \in R_{c_i}$. Then $u_iu_j \in R_{c_ic_j}$, for all
$1\leq i, j \leq 4$. Now, if $c_i$, $c_j$ are $3$-cycles in the
same conjugacy class $\mathcal O$, then the product $c_ic_j$ does
not belong to $\mathcal O$. Hence $R_{c_ic_j} = 0$, for all $1\leq
i, j \leq 4$, and $u_iu_j = 0$, for all $1\leq i, j \leq 4$. Since
$R$ is semisimple, this is a contradiction. Then $\Supp (R)
\subseteq \mathbb K$, as claimed.

\medbreak Hence $R \#k\mathbb K$ is a Hopf subalgebra of $R \#
\mathbb A_4$ of dimension $20$. Because $\mathbb K$ is normal in
$\mathbb A_4$ and $R$ is stable under the adjoint action of $H$,
this is a normal Hopf subalgebra. The proof is complete.
\end{proof}

\section{Coalgebra types in dimension $60$}\label{coal-types} In what follows $H$ will be a semisimple Hopf algebra of
dimension $60$.

\medbreak Suppose that $G(H) \neq 1$ and $H$ has an irreducible
character $\chi$ of degree $2$. Let $C \subseteq H$ be the simple
subcoalgebra containing $\chi$, and consider the Hopf subalgebra
$B = : B[\chi] = k[C\mathcal S(C)] \subseteq H_{\coad}$. Then we
have $B \simeq k^{\mathbb A_4}$, $k^{\Z_2\times\Z_2}$ or
$k^{D_n}$, $n = 3$ or $5$. See Subsection \ref{bdechi}.

\begin{lemma}\label{n=2} Suppose $G(H) = 2$. Assume that $H$ has an
irreducible character $\chi$ of degree $2$. Then $G[\chi] = G(H)$
and $B[\chi] \simeq k^{D_3}$ or $k^{D_5}$.

We have in addition:
\begin{enumerate}\item[(i)] The sum of simple subcoalgebras of dimensions $1$ and $2$
is a Hopf subalgebra of $H$. \item[(ii)] Let $B := k[B[\chi]\vert
\; \deg \chi = 2]$. Then $G(H) \subseteq G(B)\cap
Z(B)$.\end{enumerate}
\end{lemma}

\begin{proof} In this case $H$ cannot contain Hopf subalgebras isomorphic to $k^{\mathbb A_4}$ or $k^{\Z_2\times\Z_2}$.
Then $|G[\chi]| = 2$ and thus $G[\chi] = G(H)$. Since this holds
for all irreducible characters of degree $2$, part (i) then
follows from \cite[Theorem 2.4.2]{ssld}. Since $G(H) = G[\chi]$ is
contained in $B[\chi]$ for all $\chi$ of degree $2$, and because
$B[\chi]$ is commutative for all such $\chi$, then $G(H)$ is
central in $B$. This proves (ii).
\end{proof}

\begin{remark}\label{2inB} Suppose $H$ is any semisimple Hopf
algebra satisfying (ii), that is, the sum of simple subcoalgebras
of dimensions $1$ and $2$ is a Hopf subalgebra of $H$. It follows
from Lemma \ref{psiinB} that if $\chi$, $\lambda$ and $\psi$ are
irreducible characters of $H$ such that $\deg \chi, \deg \lambda
\leq 2$ and $m(\chi, \psi \lambda) > 0$, then $\deg \psi \leq 2$.
\end{remark}

\begin{proposition}\label{conteo} Suppose $H$ is not cocommutative. Then, according to the
order of $G(H)$, the coalgebra type of $H$ is one of the
following:

\begin{itemize} \item[(i)] $\vert G(H) \vert = 1$: $(1, 1;  3, 2; 4, 1; 5,
1)$. \end{itemize}

In this case, $H$ is simple and isomorphic to $k^{\mathbb A_5}$ or
to $\A_1$.

\begin{itemize} \item[(ii)] $\vert G(H) \vert = 2$: $(1, 2; 2, 1; 3, 6)$,
$(1, 2; 2, 1; 3, 2; 6, 1)$. \end{itemize} In this case the simple
subcoalgebras of dimensions $1$ and $4$ form a Hopf subalgebra of
$H$, isomorphic to $k^{D_3}$.

\begin{itemize} \item[(iii)] $\vert G(H) \vert = 2$: $(1, 2;
2, 2; 5, 2)$. \end{itemize} In this case the simple subcoalgebras
of dimensions $1$ and $4$ form a Hopf subalgebra of $H$,
isomorphic to $k^{D_5}$.

\begin{itemize} \item[(iv)] $\vert G(H) \vert = 3$:  $(1, 3; 2, 12; 3, 1)$, $(1, 3; 3,
1; 4, 3)$. \end{itemize}

\begin{itemize} \item[(v)] $\vert G(H) \vert = 4$: $(1, 4; 2,
14)$, $(1, 4; 2, 10; 4, 1)$.
\end{itemize}

\begin{itemize} \item[(vi)] $\vert G(H) \vert = 4$: $(1, 4; 2, 2; 4, 3)$.\end{itemize}
In this case the simple subcoalgebras of dimensions $1$ and $4$
form a Hopf subalgebra of $H$ of dimension $12$.

\begin{itemize} \item[(vii)] $\vert G(H) \vert = 4$: $(1, 4; 2,
6; 4, 2)$. \end{itemize} In this case, if $H$ is simple, then $H$
is isomorphic to the self-dual Hopf algebra $\B$.

\begin{itemize} \item[(viii)] $\vert G(H) \vert = 6$: $(1, 6; 2, 9; 3, 2)$, $(1, 6;
3, 6)$, $(1, 6; 3, 2; 6, 1)$. \end{itemize}

\begin{itemize} \item[(ix)] $\vert G(H) \vert = 10$: $(1,
10; 5, 2)$.\end{itemize}

\begin{itemize} \item[(x)] $\vert G(H) \vert = 12$: $(1, 12;
2, 12)$.\end{itemize}

\begin{itemize} \item[(xi)] $\vert G(H) \vert = 12$: $(1, 12;
4, 3)$.\end{itemize}

\begin{itemize} \item[(xii)] $\vert G(H) \vert = 15$: $(1, 15;
3, 5)$.\end{itemize}

\begin{itemize} \item[(xiii)] $\vert G(H) \vert = 20$: $(1, 20;
2, 10)$. \end{itemize}
\end{proposition}

We shall show in Proposition \ref{a0} below that, if $H$ is simple
and has coalgebra type $(1, 12; 4, 3)$ as in (xi), then $H$ is
isomorphic to $\A_0 \simeq \A_1^*$.

\begin{proof} By \cite[Proposition 6.10]{BN},  if $G(H) = 1$, then  $H$ is simple and isomorphic to $k^{\mathbb
A_5}$ or to $\A_1$. Also, by \cite[Corollary 6.12]{BN}, if $H$ is
simple and of coalgebra type $(1, 4; 2, 6; 4, 2)$, then $H$ is
isomorphic to the self-dual Hopf algebra $\B$.

\medbreak We shall next show that the prescribed ones are the only
possible coalgebra types. We claim that the types $(1, 3; 2, 3; 3,
5)$ and $(1, 3; 2, 3; 3, 1; 6, 1)$ are impossible. Suppose on the
contrary that $H$ is of one of these types. Then $H$ has a
self-dual irreducible character $\chi$ of degree $2$, and
necessarily $G[\chi] = 1$. Therefore, since $|G(H)| = 3$, it
follows from Theorem \ref{bn-mres}, that the Hopf subalgebra
$B[\chi]$ isomorphic to $k^{\mathbb A_4}$. In particular, $\chi
\notin B[\chi]$, and thus $B[\chi]$ has index $2$ in $k[C]$, where
$C$ is the simple subcoalgebra containing $\chi$.  This is a
contradiction since it implies that $k[C]$ is of dimension $24$,
which does not divide $\dim H$.

Also, the types $(1, 2; 2, 10; 3, 2)$ and $(1, 2; 2, 6; 3, 2; 4,
1)$ are impossible, by Lemma \ref{n=2} (i), since $H$ cannot
contain Hopf subalgebras of dimensions $42$ or $26$.

The coalgebra type $(1, 2; 2, 2; 3, 2; 4, 2)$ is not possible
neither. Indeed, in this case, the sum  of simple subcoalgebras of
dimensions $1$ and $4$ is a Hopf subalgebra $B$ of $H$ of
dimension $10$. Let $\zeta \neq \zeta' \in H$ be the irreducible
characters of degree $4$, $C_{\zeta}$, $C_{\zeta'}$,  the
corresponding simple subcoalgebras, and $C = C_{\zeta} \oplus
C_{\zeta'}$. Consider the product $\lambda \zeta$, where $\lambda$
is an irreducible character of degree $2$. Then $\lambda\zeta$
does not contain irreducible summands of degree $1$ or $2$, since
otherwise we would have $\zeta \in B$, which is a contradiction;
\textit{c.f.} Remark \ref{2inB}. Taking degrees, it follows that
$\lambda\zeta$ is a sum of irreducible characters of degree $4$.
This implies that $BC = C$. Therefore, $C$ is a $(B, H)$-Hopf
module under the action of $B$ given by left multiplication and
the coaction of $H$ given by the comultiplication. Then the
Nichols-Zoeller theorem implies that $\dim B$ divides $\dim C =
32$. This contradiction discards this possibility.

Apart from these, other than the types listed in (i)--(xiii), we
must consider the possibilities  $(1, 12; 2, 3; 6, 1)$, $(1, 12;
2, 3; 3, 4)$, with $|G(H)| = 12$, and $(1, 4; 2, 1; 4, 1; 6, 1)$,
$(1, 4; 2, 5; 6, 1)$, $(1, 4; 2, 1; 3, 4; 4, 1)$, $(1, 4; 2, 5; 3,
4)$, with $|G(H)| = 4$. In these cases, $G(H)$ contains a subgroup
of order $4$ and the number of irreducible characters of degree
$2$ is odd. Hence, by \cite[Proposition 2.1.3]{ssld}, $H$ would
contain a Hopf subalgebra of dimension $8$, which is not possible,
since $8$ does not divide $\dim H$. This discards these
possibilities and proves that these are indeed the only possible
coalgebra types.

\medbreak It follows from Lemma \ref{n=2} that the simple
subcoalgebras of dimensions $1$ and $4$ form a Hopf subalgebra of
$H$, isomorphic to $k^{D_3}$, for the coalgebra types $(1, 2; 2,
1; 3, 6)$ and $(1, 2; 2, 1; 3, 2; 6, 1)$, or to $k^{D_5}$ for the
coalgebra type  $(1, 2; 2, 2; 5, 2)$. Finally, the statement on
type $(1, 4; 2, 2; 4, 3)$ is easily seen. \end{proof}

The next two lemmas discard  the existence of certain quotient
Hopf algebras.

\begin{lemma}\label{sub-qt} Suppose $H$ has a quotient Hopf algebra of dimension $12$. Then we have:

\medbreak  \emph{(i)} If $H$ has coalgebra type $(1, 2; 2, 1; 3,
6)$, $(1, 2; 2, 1; 3, 2; 6, 1)$,  $(1, 6; 3, 6)$, or $(1, 6; 3, 2;
6, 1)$, then $H$ is not simple.

\medbreak \emph{(ii)} If $|G(H)|$ is divisible by $5$, then $H$ is
not simple. \end{lemma}

\begin{proof} Suppose $H \to B$ is a Hopf algebra quotient
with $\dim B = 12$. Then $\dim H^{\co B} = 5$. Consider  first the
case (i). Here $kG(H)\cap H^{\co B} = kG(H)^{\co B} = k1$, by
\cite{NZ}. On the other hand,  $H^{\co B}$ cannot be contained in
a Hopf subalgebra of type $(1, 2; 2, 1)$, because $\dim H^{\co B}$
does not divide $6$.

Decomposing $H^{\co B}$ into a direct sum of simple left coideals
leads to a contradiction, in view of the assumptions on the
coalgebra structure of $H$. This proves (i).

\medbreak Now suppose that $5$ divides  $|G(H)|$, so that $G(H)$
has a subgroup $F$ of order $5$. Then necessarily $kF = H^{\co
B}$, by \cite{NZ}. Thus $kF$ is normal in $H$ and thus $H$ is not
simple. This proves (ii). \end{proof}

\begin{lemma}\label{quotients}\emph{(i)} Suppose $H$ has coalgebra
type $(1, 15; 3, 5)$. Then $H$ has no quotient Hopf algebra of
dimension $6$.

\emph{(ii)} Suppose $H$ has coalgebra type $(1, 4; 2, 2; 4, 3)$.
Then $H$ has no quotient Hopf algebra isomorphic to $kD_5$.
\end{lemma}

\begin{proof}(i). Suppose on the contrary that there exists a Hopf algebra quotient $\pi: H \to L$, with $\dim L = 6$.
We have $\dim H^{\co L} = 10$. Therefore $G(H) \cap H^{\co L}$ is
of order $5$ or $1$. Decomposing $H^{\co L}$ into a direct sum of
irreducible left coideals we see that the first is impossible,
whence $|G(H) \cap H^{\co L}| = 1$. But this implies that
$\pi\vert_{G(H)}$ is injective, which contradicts \cite{NZ}, since
$|G(H)|$ does not divide $\dim L$.

\medbreak (ii). Since $H$ has $3$ irreducible characters of degree
$4$, then one of them, say $\psi$, must be a fixed element under
left multiplication by $G(H)$. Hence we have a decomposition
\begin{equation}\psi\psi^* = \sum_{g \in G(H)}g + n\lambda +
n'\lambda' + \mu,\end{equation} where $\lambda \neq \lambda'$ are
the irreducible characters of degree $2$, $\mu$ is a sum of
irreducible characters of degree $4$, and $n, n'$ are nonnegative
integers.

Since $H$ must contain a Hopf subalgebra of dimension $6$, then
left multiplication by $G(H)$  permutes transitively the set $\{
\lambda, \lambda' \}$. Then $n = n'$, because $\psi$ is fixed
under left multiplication by $G(H)$. Suppose $n \neq 0$. Then
\begin{equation}\label{lambdapsi}\lambda \psi = \psi +
\rho,\end{equation} where $\rho$ is an irreducible character of
degree $4$; otherwise, $\rho$ would contain an irreducible
character of degree $\leq 2$, implying, by Lemma \ref{psiinB},
that $\psi$ belongs to the unique Hopf subalgebra of dimension
$12$ of $H$, which is impossible.

Since $n = n'$, we may assume that $\lambda$ belongs to a Hopf
subalgebra of dimension $6$ of $H$; that is, $\lambda^2 = 1 + a +
\lambda$, where $1\neq a$ is a group-like element of order $2$.
Multiplying \eqref{lambdapsi} on the left by $\lambda$, we find
\begin{equation}3\psi + \rho = \psi + \rho + \lambda \rho. \end{equation}
Hence $\lambda \rho = 2 \psi$. Let $C_{\lambda}$, $C_{\psi}$,
$C_{\rho}$, be the simple subcoalgebras containing $\lambda$,
$\psi$ and $\rho$, respectively. Then we have $C_{\lambda}C_{\psi}
\subseteq C_{\psi} \oplus C_{\rho}$ and $C_{\lambda}C_{\rho}
\subseteq C_{\psi}$, implying that $A (C_{\psi} \oplus C_{\rho})
\subseteq C_{\psi} \oplus C_{\rho}$, where $A = k[C_{\lambda}]$ is
a Hopf subalgebra of dimension $6$. Then $C_{\psi} \oplus
C_{\rho}$ is an $(A, H)$-Hopf module. But this contradicts
\cite{NZ}, because $\dim A$ does not divide $32 = \dim (C_{\psi}
\oplus C_{\rho})$. Therefore we have $n = n' = 0$. That is,
\begin{equation}\label{solo1-4}\psi\psi^* = \sum_{g \in G(H^*)}g  +
\mu,\end{equation} where $\mu$ is a sum of irreducible characters
of degree $4$.

\medbreak Suppose on the contrary that there exists a quotient
Hopf algebra $\pi: H \to kD_5$. We have $\dim H^{\co \pi} = 6$.
Then, either $H^{\co \pi}$ contains a unique irreducible left
coideal of dimension $4$, or $H^{\co \pi} = k1 \oplus ka \oplus U
\oplus U'$, where $U$ and $U'$ are irreducible coideals of
dimension $2$ and $1\neq a \in G(H)$.

In view of \eqref{solo1-4}, the last possibility implies that
$m(1, \pi(\psi\psi^*)) = 2$, which contradicts Lemma \ref{hcopi}
(i). Therefore we may assume that
\begin{equation}\label{h*copi}H^{\co \pi} = k1 \oplus ka \oplus V,\end{equation}
as a left coideal of $H$, where $a\in G(H)$ is of order $2$ and
$V$ is an irreducible left coideal of dimension $4$. Let $\zeta
\in H$ be the irreducible character corresponding to $V$. We have
$\zeta = \zeta^*$.

\medbreak Consider the decomposition \eqref{solo1-4}. Using Lemma
\ref{zetazeta*} (i) and the decomposition \eqref{h*copi} of
$H^{\co \pi}$, we get $m(1, \pi(\mu)) \geq 2$. Hence $m(\zeta,
\psi\psi^*) = m(\zeta, \mu) \geq 2$. This implies that $|G[\zeta]|
= 4$; otherwise, since $\psi$ is stable under left multiplication
by $G(H)$, we would  also have $m(g\zeta, \psi\psi^*) \geq 2$, for
some $g \in G(H)$ such that $g\zeta \neq \zeta$, whence the
contradiction $\deg \psi \psi^* \geq 20$.

In particular, relation \eqref{solo1-4} holds for $\zeta$ in the
place of $\psi$, and thus
\begin{equation}\label{zeta^2}\zeta^2 = \sum_{g \in G(H)}g +
m\zeta + \zeta',\end{equation} where $m \geq 2$, and $\zeta'$ is
irreducible of degree $4$.

\medbreak Write $\pi(\zeta) = 1 + x + y + z$, where $x, y, z \in
D_5\backslash \{ 1\}$. Since the dimension of an induced
representation from $k^{D_5}$ to $H^*$ is $6 = |H^*: k^{D_5}|$, we
see that the multiplicity of $\zeta$ in such representation is at
most $1$. By Frobenius reciprocity,  $x$, $y$ and $z$ are pairwise
distinct. Since $\pi(G(H))$ is a subgroup of $D_5$, then
$|\pi(G(H))| = 2$. Thus $\pi(g)^2 = 1$, for all $g \in G(H)$.
Suppose $g \in G(H)$ is such that $\pi(g)\neq 1$. The relation
$g\zeta = \zeta$ implies that $m(\pi(g), \pi(\zeta))
> 0$. Hence we may assume $\pi(g) = x$.

Applying $\pi$ to the relation \eqref{zeta^2}, we get
\begin{equation}(1+ x + y + z)^2 = (2+m)(1 + x) + my + mz + \pi(\zeta'). \end{equation} Comparing the multiplicity of  $x$ on both
sides of this equality, and since $m \geq 2$, we see that $x = yz
= zy$.

On the other hand, by self-duality of $\pi(\zeta)$ and because
$x^2 = 1$, we must have $y = z^{-1}$ or $y^2 = z^2 = 1$. This is
is a contradiction since, in any case, neither $zy$ nor $yz$ can
be of order $2$. This shows that the decomposition \eqref{h*copi}
is impossible. Then $H$ cannot have quotient Hopf algebras
isomorphic to $kD_5$, as claimed.
\end{proof}

It was shown in \cite[Theorem 6.4]{BN} that for every semisimple
Hopf algebra $H$ such that all irreducible characters have degree
at most $2$, either $H$ or $H^*$ must contain a central group-like
element. As a consequence, we have:

\begin{proposition}\label{deg2} Suppose $H$ has coalgebra type $(1, 4; 2, 14)$, $(1, 12; 2, 12)$ or  $(1, 20; 2, 10)$.
Then $H$ is not simple. \qed \end{proposition}

\section{Proof of the main result}\label{pf}
In the following subsections we shall consider the distinct
possibilities for the coalgebra type of $H$, arising from
Proposition \ref{conteo}.

The results in this section, combined with Proposition
\ref{conteo}, imply the statement in Theorem \ref{main}, namely,
that the only simple semisimple Hopf algebras of dimension $60$
are exactly $\A_0$, $\A_1$ and $\B$.

\subsection{Type (vi)}\label{five} We know from Proposition \ref{deg2} that if $H$ has coalgebra type $(1, 4; 2,
14)$, then $H$ is not simple. We shall show in this subsection
that the same occurs for the type $(1, 4; 2, 10; 4, 1)$.

\begin{proposition}\label{b} Suppose $H$ is of type $(1, 4; 2, 10; 4, 1)$ as a
coalgebra. Then $H$ is not simple. \end{proposition}

\begin{proof} There must exist irreducible characters
$\chi$ and $\chi'$ of degree $2$, such that $\chi\chi'$ is
irreducible of degree $4$. Otherwise, the sum of simple
subcoalgebras of dimensions $1$ and $2$ would be a Hopf subalgebra
of $H$ of dimension $44$, which is impossible. By \cite[Theorem
2.4.2]{ssld}, we have $G[\chi] \cap G[\chi'] = 1$, thus $G[\chi]$,
$G[\chi']$ are distinct subgroups  of order $2$. If $H$ is simple,
then by \cite[Lemma 6.11]{BN}, it should be $H \simeq \B$, which
contradicts the assumption on the coalgebra type of $H$. Hence $H$
is not simple, as claimed. \end{proof}

\subsection{Type (iv)} We show in the next two propositions that there is no
simple Hopf algebra in this type.

\begin{proposition}\label{n=3} Suppose $H$ is of type $(1, 3; 2, 12; 3, 1)$ as a coalgebra.
Then $H_{\coad}$ is a commutative Hopf subalgebra of dimension
$12$. In particular, $H$ is not simple. \end{proposition}

\begin{proof} Since $|G(H)|$ is odd, then for all irreducible character $\chi$ of degree
$2$ we have $\chi \chi^* = 1 + \lambda$, where $\lambda$ is
irreducible of degree $3$. It follows that the irreducible
subcoalgebra of dimension $9$ generates a commutative Hopf
subalgebra $A$ of dimension $12$, such that $\chi\chi^* \in A$,
for all irreducible characters $\chi \in H$. See  \cite{NR},
\cite[Remark 3.4]{BN}. Then we have $H_{\coad} = A$. The
proposition follows since $H_{\coad}$ is a normal Hopf subalgebra
of $H$. \end{proof}

\begin{proposition} Suppose $H$ is of type $(1, 3; 3, 1; 4, 3)$ as a
coalgebra.  Then $H$ is not simple. \end{proposition}

\begin{proof} It is easily seen that the simple subcoalgebras of dimensions
$1$ and $9$ form a Hopf subalgebra $K$ isomorphic to $k^{\mathbb
A_4}$. Suppose that $H$ is simple. By Lemma \ref{sub-qt}, $H^*$
cannot be of any of the types (ii), (viii), (ix), or (xii).

If $\pi: H \to B$ is a Hopf algebra quotient such that $\dim B =
12$, then $\dim H^{\co B} = 5$ and thus $K\cap H^{\co B} = k1$.
Then $\pi$ restricts to an isomorphism $K \to B$. In particular,
$B^* \simeq k\mathbb A_4 \subseteq kG(H^*)$, and $H^*$ is a
biproduct $H^* \simeq R \# k\mathbb A_4$. Hence, by Proposition
\ref{braided-a4}, $H$ is not simple.

By Propositions \ref{conteo}, \ref{deg2}, \ref{b} and \ref{n=3},
$H^*$ must be of type $(1, 2; 2, 2; 5, 2)$ as a coalgebra. In
particular, there is a Hopf algebra quotient $\pi: H \to B$, with
$B\simeq kD_5$.

\medbreak Let $R = H^{\co B}$. We have $\dim R = 6$ and
necessarily $G(H) \subseteq R$. In view of the coalgebra type of
$H$, this implies that, as a left coideal of $H$,
\begin{equation}\label{r}R = kG(H) \oplus V, \end{equation} where $V$ is
an irreducible left coideal of dimension $3$.

\medbreak Write $G(H) = \{ 1, a, a^2 \}$, and let $\psi \in K$ be
the irreducible character of degree $3$. Let also $\zeta \in H$ be
an irreducible character of degree $4$.

We have a decomposition $\zeta\zeta^* = 1 + n\psi + m_1\zeta +
m_2a\zeta + m_3a^2\zeta$. Taking degrees, this implies $n \neq 0$.
This gives in turn $\psi\zeta = \zeta + a \zeta + a^2\zeta$, since
$\psi$, and thus also $\psi\zeta$, are stable under left
multiplication by $G(H)$. Hence $n = 1$.

The decomposition \eqref{r} implies that $m(1, \pi(a^i\zeta)) =
0$, for all $i = 0, 1, 2$, in view of Remark \ref{hcopi}. On the
other hand, $m(1, \pi(\psi)) = 1$. Thus we get $m(1,
\pi(\zeta\zeta^*)) = 2 < \deg \zeta = 4$. This contradicts Lemma
\ref{zetazeta*} (i). The contradiction comes from the assumption
that $H$ is not simple, hence the proposition follows. \end{proof}

\subsection{Type (viii)} By Lemma \ref{sub-qt} we may assume that $H$ has no quotient Hopf
algebra of dimension $12$. In view of Lemma \ref{quotients} (i),
$H^*$ is not of type $(1, 15; 3, 5)$.  Therefore, Proposition
\ref{conteo} and the previous results imply that there is a
quotient Hopf algebra $\pi: H \to \overline H$, where $\dim
\overline H = 10$ or $6$.

\begin{proposition} Suppose $H$ has coalgebra type $(1, 6; 2, 9; 3, 2)$.
Then $H$ is not simple. \end{proposition}

\begin{proof} If $G[\chi] \neq 1$, for all irreducible characters
$\chi$ of degree $2$, then the sum of simple subcoalgebras of
dimensions $1$ and $2$ is a Hopf subalgebra of $H$ of dimension
$42$, which is a contradiction. Then $G[\chi] = 1$ for some of
these characters. Then $B[\chi] \simeq k^{\mathbb A_4}$ of
coalgebra type $(1, 3; 3, 1)$.

On the other hand, $G(H)$ contains a unique normal subgroup $F
\simeq \Z_3$. We have $F \subseteq B[\chi] \simeq k^{\mathbb
A_4}$, with $\chi$ as before. Consider the Hopf subalgebra $K =
k[G(H), B[\chi]]$. Then $kF$ is a normal Hopf subalgebra of $K$.

Since $\dim K > 24$ and $\dim K$ is divisible by $12$, then $K =
H$ and $kF$ is normal in $H$. Therefore $H$ is not simple, as
claimed. \end{proof}

\medbreak It remains to consider the types $(1, 6; 3, 6)$ and $(1,
6; 3, 2; 6, 1)$. Let $F \subseteq G(H)$ be the unique subgroup of
order $3$. Then $F$ is the common stabilizer of all irreducible
characters of degree $3$.

\begin{proposition} Suppose $H$ is of type $(1, 6; 3, 2; 6, 1)$. Then
$H$ is not simple. \end{proposition}

\begin{proof} Let $\chi \neq \chi' \in H$ be the irreducible characters of
degree $3$. Then $g \chi = \chi = \chi g$, for all $g \in F$, and
$a \chi = \chi' = \chi a$, where $a\in G(H)$ is any element of
order $2$. Then there are decompositions
\begin{equation}\label{chichi} \chi\chi^* = \chi'(\chi')^* = \sum_{g \in F}g +
\zeta, \end{equation} where $\zeta$ is the irreducible character
of degree $6$. Otherwise, the product of irreducible characters of
degree $3$ would be a sum of irreducible characters of degree $1$
and $3$, implying that there is a Hopf subalgebra of coalgebra
type $(1, 6; 3, 2)$, which is impossible by \cite{NZ}. In
particular, $k[C] = H$ for all simple subcoalgebras $C$ of
dimension $9$.

Suppose first that there is a quotient $\pi: H \to \overline H$,
with $\dim \overline H = 10$. Then $H^{\co \overline H} = kF
\oplus V$, where $V$ is an irreducible character of degree $3$.
Let $C\subseteq H$ be the simple subcoalgebra containing $V$. By
\cite[Corollary 3.5.2]{ssld}, $kF$ is normal in $k[C] = H$. Then
$H$ is not simple.

Consider the case where there is a quotient $\pi: H \to \overline
H$, with $\dim \overline H = 6$. We first claim that $\overline H$
must be cocommutative. To see this, we consider the intersection
$H^{\co \pi} \cap kG(H)$ and the possible decompositions of
$H^{\co \pi}$ as a left coideal of $H$. We have $\dim H^{\co \pi}
= 10$. Counting dimensions  we get that $\dim H^{\co \pi} \cap
kG(H) \neq 2$. Also, $\dim H^{\co \pi} \cap kG(H)$ is not
divisible by $3$, by \cite{NZ}. Hence $H^{\co \pi} \cap kG(H) =
k1$, implying that the restriction of $\pi$ induces an isomorphism
$\pi: kG(H) \to \overline H$. Thus $\overline H$ is cocommutative,
as claimed.

\medbreak Counting dimensions, we see that the multiplicity of a
simple comodule of dimension $6$ in $H^{\co \overline H}$ can be
$1$ or $0$. Let $\zeta \in H$ be the irreducible character of such
a comodule. By Remark \ref{hcopi}, $m(1, \pi(\zeta)) = 1$ or $0$.
Combining this with the decomposition \eqref{chichi}, we get that
$m(1, \pi(\chi\chi^*)) = 2$ or $1$. This contradicts Lemma
\ref{zetazeta*}. The contradiction shows that $H$ is not simple
and  finishes the proof of the proposition. \end{proof}

\begin{proposition} Suppose $H$ is of type $(1, 6; 3, 6)$ as a coalgebra. Then $H$ is
not simple. \end{proposition}

\begin{proof} Assume on the contrary that $H$ is simple. We first claim that there
is no quotient $\pi: H \to \overline H$, with $\dim \overline H =
10$.

Suppose on the contrary that such a quotient exists. Then $H^{\co
\overline H} = kF \oplus V$, where $V$ is an irreducible character
of degree $3$. Let $C\subseteq H$ be the simple subcoalgebra
containing $V$.

Since $H^{\co \overline H}$ is a subalgebra of $H$ and $V$ is the
only $3$-dimensional irreducible left coideal contained in $H^{\co
\overline H}$, then $gV = V = Vg$, for all $g \in F$.

By \cite[Corollary 3.5.2]{ssld}, $kF$ is normal in $k[C] = H$.
Since $\dim k[C] \geq 12$ and we are assuming that $H$ is simple,
then $k[C]$ is of dimension $12$ and moreover, the coalgebra type
of $k[C]$ is $(1, 3; 3, 1)$. In particular, $k[C]$ is commutative.
Consider the Hopf subalgebra $K = k[G(H), C]$. Since $kG(H), k[C]
\subseteq K$, then $K = H$, by dimension. On the other hand, the
subgroup $kF$ is normal in $kG(H)$ and also in $k[C]$, hence $kF$
is normal in $H$. This implies the claim.

\medbreak In view of the above, it follows from Proposition
\ref{conteo} that there is a quotient $\pi: H \to \overline H$,
with $\dim \overline H = 6$. We have $kG(H) \cap H^{\co \pi} =
k1$, by counting dimensions. Thus $\pi$ restricts to an
isomorphism $kG(H) \to \overline H$, implying that $H$ is a
biproduct $H = R \# kG(H)$.

As a left coideal of $H$, we must have a decomposition $R = k1
\oplus V_1 \oplus V_2 \oplus V_3$, where $V_i$ is an irreducible
left coideal of dimension $3$, $i = 1, 2, 3$.

Since $F$ stabilizes all irreducible characters of degree $3$,
then, for all $i = 1, 2, 3$, we have $(V_i \# 1) (1 \# g) = V_ig
\simeq V_i$. In particular, for each $i$, $V_i\#kF \subseteq C_i$,
where $C_i$ is the simple subcoalgebra containing $V_i$. Hence, by
dimension, $V_i\#kF = C_i$ is a subcoalgebra of $H$. This implies
that the subalgebra $K = R\#kF \subseteq H$ is also a
subcoalgebra, hence a Hopf subalgebra. Then $H$ is not simple,
since $|H: K| = 2$. \end{proof}

\subsection{Type (xii)} In this case $H$ is of type $(1, 15; 3, 5)$ as a
coalgebra. By Lemmas \ref{sub-qt} and \ref{quotients} we may
assume that $H$ has no quotient Hopf algebra of dimension $12$ or
$6$. By Proposition \ref{conteo} and the results in the previous
subsections, we may assume that $H^*$ is of type (iii), (ix) or
(xii).

\begin{proposition}\label{15} $H$ is not simple. \end{proposition}

\begin{proof} There is a Hopf algebra quotient $\pi: H^* \to
k^{G(H)}$, and we have $\dim (H^*)^{\co \pi} = 4$. If $2$ divides
$|G(H^*)|$ and $\Gamma \subseteq G(H^*)$ is a subgroup of order
$2$, then $\Gamma \subseteq (H^*)^{\co \pi}$, by \cite{NZ}. Hence
we may assume that $(H^*)^{\co \pi} = k\Gamma \oplus V$, where $V$
is an irreducible left coideal of dimension $2$. This discards the
possibility (ix) for the coalgebra type of $H^*$.

Thus $H^*$ is of type (iii) in this case. By Proposition
\ref{conteo}, the simple subcoalgebras of dimensions $1$ and $4$
form a Hopf subalgebra $K$ of $H$, isomorphic to $k^{D_5}$. In
view of the decomposition of $(H^*)^{\co \pi}$, we have
$(H^*)^{\co \pi} \subseteq K$. But this is impossible since $\dim
(H^*)^{\co \pi} = 4$ does not divide $\dim K$.

\medbreak Therefore we may assume that $|G(H^*)|$ is odd, and thus
that $H^*$ is of type (xii) as a coalgebra. In this case
$\pi\vert_{kG(H^*)}: kG(H^*) \to k^{G(H)}$ is an isomorphism and
$H$ is a biproduct $H \simeq R \# k\Z_{15}$, where $R$ is a
Yetter-Drinfeld Hopf algebra over $\Z_{15}$ of dimension $4$. By
\cite[Proposition 4.4.6]{ssld}, $H$ is not simple.
\end{proof}

\subsection{Type (ix)} Here, $H$ is of type $(1, 10; 5, 2)$ as a
coalgebra.  By Lemma \ref{sub-qt}, we may assume that $H$ has no
quotient Hopf algebra of dimension $12$. By Proposition
\ref{conteo} and previous results, $|G(H^*)| = 2$ or $10$.

\begin{proposition}\label{10} $H$ is not simple. \end{proposition}

\begin{proof} By Proposition
\ref{conteo}, we may further assume there is a Hopf algebra
quotient $\pi: H \to B$, where $\dim B = 6$ or $10$.

If $\dim B = 6$, then $\dim H^{\co B} = 10$ and, by \cite{NZ},
$H^{\co B}\cap kG(H) = kF$, where $F$ is the unique subgroup of
order $5$ of $G(H)$. Then $H^{\co B} = kF \oplus U$, where $U$ is
an irreducible left coideal of dimension $5$. Then, for all $g \in
F$, $gV = V = Vg$, and  by \cite[Corollary 3.5.2]{ssld}, $kF$ is a
normal Hopf subalgebra in $k[C]$, where $C$ is the simple
subcoalgebra containing $U$. Since $F$ is also normal in $G(H)$,
then $kF$ is normal in $k[G(H), C] = H$. Hence $H$ is not simple
in this case.

Finally, suppose $\dim B = 10$. Then  $H^{\co B} \cap G(H) = k1$,
in view of \cite{NZ} and the coalgebra type of $H$. Then $\pi$
induces an isomorphism $kG(H) \simeq B$. Thus $H$ is a biproduct
$H = R \# kG(H)$, where $R$ is a $6$-dimensional Yetter-Drinfeld
(braided) Hopf algebra over $G(H)$. Moreover, $R \simeq
H/HkG(H)^+$ as coalgebras. Since the stabilizer of a simple
subcoalgebra of $H$ is cyclic of order $5$, then $R$ is
cocommutative, by Remark \ref{coalg-r} and \cite[Remark 3.2.7 and
Corollary 3.3.2]{ssld}.

The action of $G(H)$ permutes the $5$ nontrivial group-like
elements of $R$. If $G(H)$ is cyclic, then it contains a
nontrivial subgroup $F$ acting trivially on $R$ (since $\mathbb
S_5$ does not have elements of order $10$). Then $kF$ would be a
normal Hopf subalgebra of $H$ \cite[Lemma 4.4.4]{ssld}. Therefore
we may assume that $G(H) \simeq D_5$. Now the result follows from
Proposition \ref{braided-dih}. \end{proof}

\subsection{Type (xi)} The following proposition says that the simple Hopf algebra $\mathcal A_0 \simeq (k\mathbb A_5)^J$ is indeed
characterized by its coalgebra type. This has already been shown
for the other two simple examples, $\A_1$ and $\mathcal B$, in
\cite{BN}.

\begin{proposition}\label{a0} Suppose $H$ is of type $(1, 12; 4, 3)$ as a
coalgebra. If $H$ is simple, then $H \simeq \A_0$.
\end{proposition}

\begin{proof} Suppose that $H$ is simple.
If $G(H^*) = 1$, then we know from \cite[Proposition 6.10]{BN}
that $H^* \simeq \A_1$. Hence $H \simeq \A_0$. So we may assume
that there is a proper Hopf algebra quotient $\pi: H \to B$.

By Proposition \ref{conteo}, Lemma \ref{sub-qt}, and the previous
results, we may further assume that $\dim B = 10$ or $12$. The
first possibility implies that a subgroup $F \simeq \Z_3$ of
$G(H)$ must be contained in $H^{\co B}$. Since $\dim H^{\co B} =
6$, and $H$ has no irreducible left coideals of dimension $3$,
then $H^{\co B} \subseteq kG(H)$ is a normal Hopf subalgebra of
$H$.

Then $\dim B = 12$ and we see, after decomposing $H^{\co B}$ into
a sum of irreducible left coideals, that $\pi\vert_{kG(H)} : kG(H)
\to B$ must be an isomorphism. Then $H$ is a biproduct $R\#kG(H)$,
where $R$ is a Yetter-Drinfeld Hopf algebra of dimension $5$ over
$G(H)$. By Proposition \ref{braided-a4}, we may assume that $G(H)$
is not isomorphic to $\mathbb A_4$.

\medbreak If $R$ is cocommutative, the action of $G(H)$ on $R$
being by coalgebra automorphisms, must permute the set
$G(R)\backslash \{ 1\}$ of nontrivial group-likes in $R$. Thus, it
induces a group homomorphism $\theta: G(H) \to \mathbb S_4$. The
group algebra of the kernel of $\theta$ is a normal Hopf
subalgebra of $H$ \cite[Lemma 4.4.4]{ssld}. Then  $G(H)$ acts
faithfully on $G(R)\backslash \{ 1\}$. Therefore, $G(H)$ is
isomorphic to $\mathbb A_4$ (the only subgroup of $\mathbb S_4$ of
order $12$), against our assumption. Thus  we may assume that $R$
is not cocommutative.

As a left coideal of $H$, $R = k1 \oplus U$, where $U$ is an
irreducible left coideal of dimension $4$. Let $C \subseteq H$ be
the simple subcoalgebra containing $U$, and let $\Gamma \subseteq
G(H)$ be the stabilizer of $C$: that is, $Cg = C$, for all $g \in
\Gamma$. So that $\Gamma$ is of order $4$ and $\Gamma$ is not
cyclic, otherwise $R$ would be cocommutative, in view of
\cite[Remark 3.2.7 and Corollary 3.3.2]{ssld}.

For all $g \in \Gamma$, we have $U \# g = Ug \simeq U$. Hence $U
\# g \subseteq C$, and therefore, by dimension, $C = U \#
k\Gamma$. By Remark \ref{coalg-r}, the coaction of $kG(H)$ on $R$
is given by $\rho = (\epsilon_R \otimes \id)\Delta: R \to kG(H)
\otimes R$. Then $\rho(U) \subseteq k\Gamma \otimes R$, and thus
$\rho(R) \subseteq k\Gamma \otimes R$. By \cite[Lemma
4.3.1]{ssld}, $K = R\# k\Gamma$ is a Hopf subalgebra of $H$ of
dimension $20$. The coalgebra structure of $H$ forces $K$ to be
commutative.

If $\Gamma$ is normal in $G(H)$, then $K$ is normal in $H$, since
$K$ and $G(H)$ generate $H$ as an algebra. Then we can assume that
$\Gamma$ is not normal, and therefore $G(H)$ is a semidirect
product $G(H) = T \rtimes \Gamma$, where $T$ is a subgroup of
order $3$, with respect to an action $\Gamma \to \Aut T  \simeq
\Z_2$ by group automorphisms. This implies that $\Gamma$ has an
element $g\neq 1$ which is central in $G(H)$. Then $g$ is central
in $H$, because $g \in K$, which is commutative, and $K$ and
$G(H)$ generate $H$ as an algebra. This finishes the proof of the
proposition.
\end{proof}

\subsection{Type (ii)}\label{case-2} Let $B \simeq k^{D_3} \subseteq H$ be
the unique Hopf subalgebra of dimension $6$.

\begin{lemma}\label{2-1} Suppose $H$ is simple. Then $H^*$ is of type (iii) as a coalgebra. \end{lemma}

\begin{proof} In view of previous results, the possible types for $H^*$ can be (ii), (iii) and (vi).
Type (vi) is discarded by Lemma \ref{sub-qt}. Thus it is enough to
discard the possibility of $H^*$ being also of type (ii). Suppose
on the contrary that this occurs. Then there is a quotient Hopf
algebra $\pi: H \to kD_3$, and we have $\dim H^{\co \pi} = 10$.
Counting dimensions in the possible decompositions of $H^{\co
\pi}$ as a left coideal of $H$, we see that $B^{\co \pi} = B \cap
H^{\co \pi} = k1$, since $\dim B^{\co \pi}$ must divide $\dim B$.

Hence $\pi$ restricts to an isomorphism $B \to kD_3$, which is a
contradiction, because $B$ is not cocommutative. Thus $H$ has no
quotient Hopf algebra isomorphic to $kD_3$, and the lemma is
proved. \end{proof}

\begin{proposition} $H$ is not simple. \end{proposition}

\begin{proof}  Let $\chi \in H$ be an irreducible character of degree $3$.
Then there is a decomposition $\chi \chi^* = 1 + n \lambda + \mu$,
where $\lambda \in B$ is the irreducible character of degree $2$,
and $\mu$ is a sum of irreducible characters of degrees $3$ or
$6$. In particular, $n \neq 0$. Moreover, since for $1 \neq a \in
G(H)$ we have $a\lambda = \lambda$,  then $\lambda\chi = \chi + a
\chi$. Therefore, since $a\chi \neq \chi$, we find $n = m(\lambda,
\chi \chi^*) = m(\chi, \lambda\chi) = 1$. Hence,
\begin{equation}\label{dec-chi}\chi \chi^* = 1 + \lambda + \mu.
\end{equation}
This implies that $B \subseteq k[C]$, where $C$ is the simple
subcoalgebra containing $\chi$. In addition, $\dim k[C] \geq \dim
B + \dim C = 15$. Since $6 = \dim B$ divides $\dim k[C]$, then
$\dim k[C] = 30$ or $k[C] = H$. If $\dim k[C] = 30$, then $k[C]$
is normal in $H$ and we are done. Thus we can consider the case
where $k[C] = H$, for all simple subcoalgebras $C \subseteq H$ of
dimension $9$.

\medbreak By Lemma \ref{2-1}, we may assume that there is a
quotient $\pi: H \to kD_5$. The left coideal subalgebra $H^{\co
\pi}$ has dimension $6$. Hence, unless $H^{\co \pi} = B$, in which
case we are done, we may assume that $H^{\co \pi} = k1 \oplus U
\oplus V$ as a left coideal of $H$, where $U$ is an irreducible
coideal of dimension $2$ and $V$ is an irreducible coideal of
dimension $3$. Let $\chi \in H$ be the irreducible character
corresponding to $V$, and $C$ the simple subcoalgebra containing
$\chi$. Since $V$ is the only $3$-dimensional irreducible left
coideal contained in the self-dual left coideal $H^{\co \pi}$, we
have $\chi^* = \chi$.

By Remark \ref{hcopi}, $m(1, \pi(\chi)) = 1$. Also, $m(1,
\pi(\chi^2)) = m(1, \pi(\chi\chi^*)) \geq 3$, by Lemma
\ref{zetazeta*} (i).

Write $\pi(\chi) = 1 + x + y$, where $x, y \in D_5$, $x\neq 1 \neq
y$. Then $\pi(k[C]) \subseteq \langle x, y \rangle$, and therefore
$\langle x, y \rangle = D_5$, because $k[C] = H$. In particular,
$x\neq y, y^{-1}$. Furthermore, $m(1, \pi(\chi^2)) \geq 3$, and
$\pi(\chi^2) = (1 + x + y)^2$, hence $x^2 = y^2 = 1$.

Let $\pi(\lambda) = 1 + t$, with $1\neq t \in D_5$. Since
$a\lambda = \lambda$, then $\pi(a) = t$ and $t^2 = 1$. In view of
the decomposition \eqref{dec-chi}, we have $m(t, \pi(\chi^2)) =
m(t, \pi(\chi\chi^*)) > 0$. Hence, $t = x$, $y$ or $xy$. In the
first two cases, we find that $m(1, \pi(a\chi)) > 0$, which is
impossible since $a\chi \neq \chi$, and $\chi$ is the only
irreducible character of degree $3$ appearing in $H^{\co \pi}$.
Therefore $t = xy$. But, since $x$ and $y$ are reflections in
$D_5$, then the order of $xy$ divides $5$. Thus $t = 1$. This
implies that $a \in H^{\co \pi}$, against our assumption. Hence we
conclude that $H$ is not simple. \end{proof}

\subsection{Type (iii)} Let $k^{D_5} \simeq B \subseteq H$ be its (unique) Hopf subalgebra of
dimension $10$, which has coalgebra type $(1, 2; 2, 2)$.

\begin{proposition}\label{n=5} $H$ is not simple. \end{proposition}

\begin{proof} In view of the previous results, if $H$ is simple, then $H^*$ is of type (iii).
But we shall show that $H$ admits no Hopf algebra quotient $\pi: H
\to \overline H$, with $\overline H \simeq kD_5$. This will imply
the proposition. Suppose on the contrary that such quotient
exists. We have $\dim H^{\co \overline H} = 6$. The coalgebra
structure of $H$ forces $H^{\co \overline H} \subseteq B$ or
$H^{\co \overline H} \cap B = k1$. However, since $6$ does not
divide $\dim B$, then $H^{\co \overline H} \subsetneq B$. Also,
$H^{\co \overline H} \cap B \neq k1$, because otherwise $\pi$
would induce an isomorphism $k^{D_5} \to kD_5$, which is
impossible. \end{proof}

\subsection{Type (vi)}\label{final} Suppose that $H$ is a semisimple
Hopf algebra of dimension $60$ which is simple as a Hopf algebra.
In view of the results in Section \ref{coal-types}, unless $H$ is
isomorphic to $\mathcal A_0$, $\mathcal A_1$ or $\mathcal B$, then
$H$ and $H^*$ are both of type $(1, 4; 2, 2; 4, 3)$ as coalgebras.

We shall show in this subsection, \textit{c. f.} Proposition
\ref{3-casos}, that this cannot occur, that is, such a semisimple
Hopf algebra cannot be simple. This will conclude the proof of
Theorem \ref{main}.

\medbreak Let $B \subseteq H$ be the (unique) Hopf subalgebra of
dimension $12$, which has coalgebra type $(1, 4; 2, 2)$.

\begin{lemma}\label{dim20} Suppose $H$ contains a Hopf subalgebra $K$ of
dimension $20$. Then $G(H)\cap Z(H) \neq 1$. \end{lemma}

\begin{proof} In view of the coalgebra structure of $H$, $K$ must be
commutative and $G(H) \subseteq K$. Let $1\neq g \in G(H)$ be a
central group-like element of $B$. Such central group-like exists
in view of the classification of semisimple Hopf algebras of
dimension $20$ \cite{pqr}. Since $G(H) \subseteq K$, then $g$ is
central in $K$ and therefore $g$ is central in $k[B, K]$. On the
other hand, $k[B, K] = H$, by dimension. Hence $G(H)\cap Z(H) \neq
1$, as claimed. \end{proof}

\begin{lemma}\label{bip} Assume that $H$ is simple. Then $H$ is a
biproduct $H \simeq R \# B$, where $R$ is a Yetter-Drinfeld Hopf
algebra over $B$ of dimension $5$. \end{lemma}

\begin{proof} Since $H^*$ is also of type $(1, 4; 2, 2; 4, 3)$, then
there is a Hopf algebra quotient $q: H \to B'$, where $B'$ is a
Hopf algebra of dimension $12$, such that $(B')^* \subseteq H^*$
is of coalgebra type $(1, 4; 2, 2)$. Since $\dim H^{\co q} = 5$,
then $\dim H^{\co q}  \cap B = k1$. Thus $q\vert_B: B \to B'$ is
an isomorphism, and $H$ is a biproduct, as claimed. \end{proof}

\begin{proposition}\label{3-casos} $H$ is not simple. \end{proposition}

\begin{proof} The proof will follow from Lemmas \ref{b-conm}, \ref{a_0} and \ref{a_1} below.
\end{proof}

\begin{lemma}\label{b-conm} Suppose $B$ is commutative. Then $H$ is not simple. \end{lemma}

\begin{proof} By Lemma \ref{bip}, $H^*$ and $H$ have the same coalgebra type and $H \simeq R \# B$ is a
biproduct, where $R$ is a Yetter-Drinfeld Hopf algebra over $B$ of
dimension $5$. If $B$ where commutative, then $B^* \subseteq H^*$
would be a cocommutative Hopf subalgebra of dimension $12$, which
is not possible. \end{proof}

Combining Lemma \ref{b-conm} with the classification of semisimple
Hopf algebras of dimension $12$ \cite{fukuda}, we may assume that
$B \simeq A_0$ or $A_1$, where $A_0$ and $A_1$ are the nontrivial
semisimple Hopf algebras of dimension $12$ such that $G(A_0)
\simeq \Z_2 \times \Z_2$ and $G(A_1) \simeq \Z_4$. See
\cite[5.2]{ssld}.

\begin{lemma}\label{a_0} Suppose $B \simeq A_0$. Then $H$ is not
simple. \end{lemma}

\begin{proof} By \cite[Proposition 5.2.1]{ssld}, $B \simeq A_0$ is a twisting of the group
$G = \Z_3 \rtimes (\Z_2 \times \Z_2)$ corresponding to the action
by group automorphisms of $\Z_2 \times \Z_2$ on $\Z_3$ defined by
$s.a = a^2$ and $t . a = a^2$, where $\Z_3 = \langle a|\, a^3 = 1
\rangle$, $\Z_2 \times \Z_2 = \langle s, t|\, s^2 = t^2 = 1
\rangle$. Therefore, there exists an invertible twist $J \in B
\otimes B$ such that $B^J \simeq kG$.

Consider the twisting $H^J$ of $H$. Since $B^J \subseteq H^J$ is a
Hopf subalgebra, then $G$ is isomorphic to a subgroup of $G(H^J)$
and, in particular, $|G(H^J)|$ is divisible by $12$. Moreover, we
may assume that $|G(H^J)| = 12$, and therefore $G(H^J) \simeq G$.
Otherwise $H^J$ would be cocommutative, hence a group algebra,
implying that $H$ is not simple by \cite[Theorem 4.10]{gn}.

By Proposition \ref{conteo}, $H^J$ is of type $(1, 12; 2, 12)$ or
$(1, 12; 4, 3)$ as a coalgebra. Since $G$ is not isomorphic to
$\mathbb A_4$, then $H^J$ cannot be isomorphic to the simple Hopf
algebra $\mathcal A_0$. Therefore, by Propositions \ref{deg2} and
\ref{a0}, $H^J$ is not simple.

\medbreak Consider first the possibility $(1, 12; 4, 3)$ for the
coalgebra type of $H^J$. Let $K \subseteq H^J$ be a proper normal
Hopf subalgebra. Then $K \subseteq kG(H^J)$. Otherwise, $\dim K =
20$ and $\dim H^J/H^JK^+ = 3$. Hence $(H^J)^*$ contains a
group-like element of order $3$, which is impossible since $H^*
\simeq H$ as coalgebras.

Thus $K = k\Gamma$ for a normal subgroup $\Gamma$ of $G(H^J) = G$,
and $|\Gamma|$ is either $12$, $6$, $3$, or $2$. If $|\Gamma| =
12$, then $k\Gamma = kG$ and thus $J^{-1} \in k\Gamma \otimes
k\Gamma$. Hence $B \simeq A_0 \simeq (kG)^{J^{-1}}$ is a normal
Hopf subalgebra of $H$ \cite[Lemma 5.4.1]{ssld}.

If $|\Gamma| = 6$, there is a (not necessarily normal) quotient
Hopf algebra $\pi: H \to \overline H$ with $\dim \overline H =
10$, implying that $H^*$ contains a Hopf subalgebra of dimension
$10$. This is impossible because of the coalgebra type of $H^*$.

Similarly, if $|\Gamma| = 3$, then there is a Hopf subalgebra $A
\subseteq H^*$ with $\dim A = 20$. By Lemma \ref{dim20}, $H^*$ and
thus also $H$, are not simple.

\medbreak Consider next the type $(1, 12; 2, 12)$. By
\cite[Corollary 6.3]{BN}, either $H^J$, and thus also $H$, has a
nontrivial central group-like element, or there is a cocentral
exact sequence $k \to K \to H^J \to kU(\mathcal C)$, where
$U(\mathcal C)$ is the universal grading group of the category
$H^J$-comod of finite dimensional $H^J$-comodules, and $K =
(H^J)_{\rm{coad}} \subsetneq H^J$.

Since $(H^J)^* = H^*$ as coalgebras, then $\widehat{U(\mathcal C)}
= G(k^{U(\mathcal C)})$ is of order $2$ or $4$. Therefore
$(H^J)^*$ has a central group-like element of order $2$. Hence,
there is a normal Hopf subalgebra $L\subseteq H^J$ with $\dim L =
30$. By \cite[Theorem 2]{ssld}, $L$ is necessarily commutative.
Since $H^J$ is a bicrossed product $H^J \simeq L
{}^{\tau}\#_{\sigma}k\Z_2$,  in view of the description of the
irreducible modules in the proof of \cite[Theorem 2.1]{MoW},
$(H^J)^*$ must be of type $(1, n; 2, m)$ as a coalgebra, which is
a contradiction. This discards the type $(1, 12; 2, 12)$ for the
coalgebra structure of $H^J$ and finishes the proof of the lemma.
\end{proof}

\begin{lemma}\label{a_1} Suppose $B \simeq A_1$. Then $H$ is not
simple. \end{lemma}

\begin{proof} We may assume that $H$ is a biproduct $H = R \#
A_1$. As a left coideal of $H$, we must have a decomposition $R =
k1 \oplus V$, where $V$ is an irreducible left coideal of
dimension $4$. Let $\zeta \in H$ be the character of $V$. Then
$\zeta = \zeta^*$.

\medbreak Suppose that $G(H)$ stabilizes $\zeta$. Then $(V \# 1)
(1 \# g) = Vg \simeq V$, for all $g \in G(H)$. In particular,
$V\#kG(H) \subseteq C$, where $C$ is the simple subcoalgebra
containing $V$. Hence, by dimension, $V\#kG(H) = C$ is a
subcoalgebra of $H$. This implies that the subalgebra $K =
R\#kG(H) \subseteq H$ is also a subcoalgebra, hence a Hopf
subalgebra of $H$. By Lemma \ref{dim20}, $H$ is not simple in this
case.

\medbreak Therefore we may assume that $\zeta$ is not stable under
left multiplication by $G(H)$. Note that $H$ does contain a unique
irreducible character $\psi$ of degree $4$ which is stable under
left multiplication by $G(H)$. Let $G(H)\zeta = \{ \zeta, \zeta'
\}$.

\begin{claim} We have
\begin{flalign}\label{rel2}&g\psi g^{-1} = \psi, \text{ for all } g \in G(H),&  \\
\label{rel1}&\psi\psi^* = \sum_{g \in G(H)}g + \zeta + \zeta' +
\psi.&
\end{flalign} \end{claim}

\begin{proof}[Proof of the claim] We have $|G[g\psi g^{-1}]| = 4$, for all $g \in G(H)$. Then
$g\psi g^{-1} = \psi$, for all $g \in G(H)$, since this is the
only stable irreducible character of degree $4$. This proves
\eqref{rel2}.

On the other hand, we have $\psi\psi^* = \sum_{g \in G(H)}g +
n\lambda + n'\lambda' + m\zeta + m'\zeta' + r\psi$, where $\lambda
\neq \lambda'$ are the irreducible characters of degree $2$ and
$n, n', m, m', r$ are nonnegative integers. Since $\psi$ is stable
and $G(H)\lambda = \{ \lambda, \lambda'\}$, we find that $n = n'$.
Suppose $n \neq 0$. Then $m(\psi, \lambda\psi) = m(\lambda,
\psi\psi^*) = n \neq 0$. Hence $\lambda \psi = \psi + \rho$, where
$\rho$ is irreducible of degree $4$; indeed, $\lambda\psi$ cannot
have irreducible summands of degrees $1$ or $2$, since otherwise
$\psi \in B$, which is impossible; \textit{c. f.} Lemma
\ref{psiinB}.

It follows from \eqref{rel2} that $\psi$ is stable also under
right multiplication by $g \in G(H)$, and it is, moreover, the
only irreducible character with this property. Then $\rho$ is
stable under right multiplication by $G(H)$, implying that $\rho =
\psi$. Then $n = 2$ and we have a relation $\lambda \psi = 2
\psi$. Let $A \subseteq B$ be the smallest Hopf subalgebra of $H$
containing $\lambda$; then $\dim A = 6$ or $12$. By the above,
$AC_{\psi} = C_{\psi}$, where $C_{\psi}$ is the simple
subcoalgebra containing $\psi$. This contradicts \cite{NZ}, since
$\dim C_{\psi} = 16$ is not divisible by $\dim A$.

Therefore $n = n' = 0$, and $\psi\psi^* = \sum_{g \in G(H)}g +
m\zeta + m'\zeta' + r\psi$. Since $\psi$ is stable, we have $m =
m'$. Moreover, $m+m' \neq 0$, because otherwise $\psi$ and $G(H)$
would span a standard subring of the character ring of $H$
corresponding to a Hopf subalgebra of dimension $20$, implying
that $H$ is not simple, by Lemma \ref{dim20}. Taking degrees we
find $m = m' = 1$, hence also $r = 1$. This proves \eqref{rel1}.
\end{proof}

Consider the projection $\gamma: H \to k^{G(H^*)}$. Then $\dim
H^{\co \gamma} = 15$, implying that $H^{\co \gamma}\cap G(H) = 1$
and $H^{\co \gamma} \cap A_1 = k1 \oplus U$, where $U$ is an
irreducible left coideal of dimension $2$. In particular, $H$ is a
biproduct $\tilde R \# kG(H)$, with $\tilde R = H^{\co \gamma}$, a
braided Hopf algebra over $G(H)$.

On the other hand, $R \subseteq \tilde R$. Hence, as a left
coideal of $H$, $\tilde R = k1 \oplus U \oplus V \oplus V_1 \oplus
V_2$, where $V_i$ are irreducible left coideals of dimension $4$.

\medbreak Let $C$ be the simple subcoalgebra containing $V$. If
$V_1$ and $V_2$ are both contained in $CG(H)$, then $\tilde R
\subseteq A_1 \oplus CG(H)$, and thus $H = \tilde R \# kG(H)
\subseteq A_1 \oplus CG(H)$. This is not possible because $\dim
A_1 \oplus CG(H) < \dim H$. Hence we may assume that $V_1$ is the
only, up to isomorphisms, stable irreducible left coideal.
Therefore $V_1 \# kG(H) \subseteq  C_1$, where $C_1$ is the simple
subcoalgebra of $H$ containing $V_1$. By dimension, we have $V_1
\# kG(H) = C_1$. This implies that $V_1 = (\id \otimes \epsilon)
(C)$ is a subcoalgebra of $R$.

\begin{claim}\label{afirm} The multiplicity of $V_1$ as a direct summand of $\tilde R$
equals $1$. \end{claim}

\begin{proof}[Proof of the claim] We know that $V_1$ appears in $\tilde
R$ with positive multiplicity. Since $V_1$ is not isomorphic to
$V$, it will be enough to show that the multiplicity of $V_1$ in
$\tilde R$ is not equal to $2$. Suppose on the contrary that this
is the case. Then $H = \tilde R \# kG(H) \subseteq A_1 \oplus V \#
kG(H) \oplus C_1$, where $C_1$ is the simple subcoalgebra
containing $V_1$. Counting dimensions, we see that this is not
possible. Hence the multiplicity of $V_1$ is $1$, as claimed.
\end{proof}

\medbreak It follows from Claim \ref{afirm} and \eqref{rel2} that
$V_1 \subseteq \tilde R$ is also a submodule under the (adjoint)
action of $G(H)$. Let $R_0 = k[V_1]$ be the subalgebra of $\tilde
R$ generated by $V_1$. Since $V_1$ is also a subcoalgebra of
$\tilde R$, then $R_0 \# kG(H) \subseteq \tilde R \# kG(H)$ is a
Hopf subalgebra.  If $\dim R_0 \# kG(H) = 20$, then we are done by
Lemma \ref{dim20}. Otherwise $R_0 \# kG(H) = H$. By
\cite[Corollary 1.3.2]{pqq2}, since $G(H)$ is cyclic, $\tilde R$
is a cocommutative coalgebra.

\medbreak  We claim that the (adjoint) action of $G(H)$ permutes
the set $G(V_1)$ transitively. Indeed, if this were not the case,
the centralizer $G(H)_x$ of $x$ in $G(H)$ would be of order $2$,
for all $x\in G(V_1)$, because $|G(V_1)| = 4$. Then the only
subgroup of order $2$ of $G(H)$ would centralize $V_1$ and
\textit{a fortiori} all of $\tilde R = k[V_1]$. Then this subgroup
would be central in $H$, contradicting the simplicity of $H$. This
proves the claim.

In particular, $\dim V_1^{G(H)} = 1$. We may now apply
\cite[Proposition 1.4.3]{pqq2} to conclude that $\tilde R$
contains a nontrivial group-like element of $H$, and arriving thus
to a  contradiction. This shows that $H$ is not simple and
finishes the proof of the lemma. \end{proof}

\bibliographystyle{amsalpha}

\end{document}